\documentclass[12pt,leqno]{amsart}

\usepackage{amsfonts,amsmath,amsthm}
\usepackage{amssymb,epsfig}
\usepackage{enumerate} 
\usepackage{comment}

\usepackage{tikz}
\usetikzlibrary{decorations.markings, arrows.meta}

\usepackage[text={425pt,650pt},centering]{geometry}
\usepackage{graphicx}
\usepackage{epsfig}
\usepackage{tikz}
\usepackage{caption}
\usepackage{color} %color
\definecolor{vert}{rgb}{0,0.6,0}

\numberwithin{figure}{section}

\theoremstyle{plain}
\newtheorem{thm}{Theorem}[section]

\newtheorem{lem}[thm]{Lemma}

\newtheorem{prop}[thm]{Proposition}
\theoremstyle{remark}
\newtheorem{rem}{\bf{Remark}}
\numberwithin{equation}{section}

\newcommand{\N}{\mathbb{N}}

\newcommand{\R}{\mathbb{R}}

\newcommand{\T}{\mathbb{T}}
\newcommand{\Z}{\mathbb{Z}}

\newcommand{\cA}{\mathcal{A}}

\newcommand{\X}{\mathbb{X}}
\newcommand{\Y}{\mathbb{Y}}

\newcommand{\AC}{{\rm AC\,}}

\newcommand{\USC}{{\rm USC\,}}
\newcommand{\LSC}{{\rm LSC\,}}
\newcommand{\Li}{L^{\infty}}
\newcommand{\Lip}{{\rm Lip\,}}

\newcommand{\lam}{\lambda}

\newcommand{\Div}{{\rm div}\,}

\DeclareMathOperator*{\limsups}{limsup^\ast}
\DeclareMathOperator*{\liminfs}{liminf_\ast}

\begin{document}

\title[large-time behavior]
{Large-time Behavior for a Cutoff Level-Set Mean Curvature G-equation with Source term}

\author[A.D. Calderon]
{Adrian D. Calderon}

\thanks{
The work of AC was partially supported by the RTG: Analysis and Partial Differential Equations at the University of Wisconsin, NSF DMS-2037851.
}

\address[Adrian D. Calderon]
{
Department of Mathematics, 
University of Wisconsin Madison, 480 Lincoln Drive, Madison, WI 53706, USA}
\email{adrian.calderon@.wisc.edu}

\date{\today}

\begin{abstract}
    We consider a cutoff level-set mean curvature G-equation with a non-negative source term. In particular, we study the large-time behavior of this fully nonlinear degenerate parabolic partial differential equation in two settings: periodic and radially symmetric. Due to the non-coercivity and non-convexity of the Hamiltonian, we are not able to use the standard Hamilton-Jacobi theory and instead rely on the inherent structure to find a monotonicity property and corresponding uniqueness set. Lastly, we discuss the local regularity of the solutions, as well as a representation formula in the radial symmetric setting. 
\end{abstract}

\maketitle

%%%%%%%%%%%%%%%%%%%%%%%%%%%%%%%%%%%%%%%%%%%%%%%%%%%%%%%%%%%%%%%%%

\section{Introduction and Outline of Paper}
We consider the fully nonlinear degenerate parabolic partial differential equation of the form
\begin{equation}\tag{C} \label{cauchy}
\left\{
\begin{aligned}
    &u_t + \left(-\Div\left(\frac{Du}{|Du|}\right)+1\right)_+|Du| +\vec W(x)\cdot Du = f(x) &&\text{ in } \R^n\times(0,\infty)\\
    &u(x,0) = g(x)  &&\text{ on } \R^n,
\end{aligned}
\right.
\end{equation}
where $g,f\in C(\R^n)$ are a given initial data and non-negative source, respectively. The cutoff function, which is the source of full nonlinearity, is defined by $(\cdot)_+:=\max\{0,\cdot\}$. When describing front propagation, the cutoff operator is a mathematical tool that prohibits any \emph{retreating} movement of the evolving front. We also introduce a \emph{wind-type} vector field $\vec W$, which can be an incompressible mixing flow. This term appears in the so-called G-equation, which is used in premixed turbulent combustion applications \cite{combustion,williams}. Equation \eqref{cauchy} not only has motivations arising from turbulent combustion theory, but also from modeling of crystal growth. See below in the Motivations subsection. Lastly, when describing front propagation, equation \eqref{cauchy} can be considered as a level set equation. That is, the zero-level set of the solution corresponds to the hypersurface of the front. See \cite{lvl set method, osher} for more information on the level set method.

The well-posedness theory of \eqref{cauchy} is already established, see \cite{well-posedness, user, evans and spruck, giga surface evolution}. We will consider finer properties, such as large-time behavior and regularity, under the viscosity solution framework. Such studies for similar first-order and second-order Hamilton-Jacobi equations with convex and coercive Hamiltonians have been completed in the past three decades. See \cite{namah-roquejoffre}, for example. We emphasize that the structure of the nonlinearity in \eqref{cauchy} is non-coercive with respect to the gradient term, which creates some challenges in establishing a priori gradient estimates. Moreover, there is no convexity with respect to the gradient term, which does not allow for the use of standard large-time behavior results in the Hamilton-Jacobi equation theory. Despite these challenges, we have monotonicity of the nonlinear term, namely its non-negative definiteness. Using this structure, we are able to study these deeper properties of viscosity solutions. In this work, we focus on studying \eqref{cauchy} in two settings: \emph{periodic} and \emph{radially symmetric}.

\subsection{Motivations}
From combustion theory, an interesting model arises: the cutoff G-equation with a shear-flow structure, \cite{curvature in shear flow,markstein1, bifurcations}. This model is used to describe the motion of a flame front under some \emph{wind} with the physical restriction that once a region has been burned, it cannot be burned again. This behavior is apparent in wild forest fires, for example. The flame front propagation can be described by the $(n-1)$-dimensional minimal surface-type equation 
\begin{equation*}
    v_t + \left(-\Div\left(\frac{Dv}{\sqrt{p_{n}^2+|Dv|^2}}\right)+1\right)_+\sqrt{p_{n}^2+|Dv|^2} = -p_n \tilde f(x),
\end{equation*}
see \cite{bifurcations}. We can think of \eqref{cauchy} as the $n$-dimensional non-graph mean curvature flow of the flame propagation with the cutoff function prohibiting non-physical burning of already burned regions. Studying the periodic case for the stationary problem corresponds to the ergodic cell problem in homogenization and large-time dynamics theory. Finer properties of the solutions to this model have been studied by the authors of \cite{bifurcations}. In particular, they have proved the existence of an \emph{effective burning velocity} for a range of flow intensities. This is also known as the large-time average. In fact, they showed that if the intensity of the shear-flow is greater than a critical value, there is no asymptotic speed for the evolving flame front, and therefore no hope for convergence to an asymptotic profile. See Section $1$ of \cite{bifurcations} for the main results regarding this. The question whether there is convergence in the suitable range of flow intensities is still open. However, this paper can provide a positive result if a source term appears in the equation. For more works regarding turbulent combustion as well as a two-player game approach, we refer the reader to \cite{yifeng1, turbulent velocity, game theory, Jack Xin}. For a periodic homogenization result for this model, see \cite{jiwoong}. Lastly, see \cite{simulations} for simulations of front propagations.

On the other hand, \eqref{cauchy} has motivations arising from crystal growth applications. For example, the following birth and spread model was introduced in \cite{birth and spread} and \cite{asymptotic speed}:
\begin{equation*}
w_t - \left( \Div \left( \frac{ Dw}{|Dw|} \right) +1 \right)|Dw| = \hat f(x),
\end{equation*}
where we view the unknown $w$ as the height profile of a crystal with deposition being controlled by the source term and horizontal spread of the crystal controlled by the curvature. In this model, the source was given by either a characteristic function or Lipschitz function with compact support. Under such assumptions, the authors proved the existence of the asymptotic crystal growth speed
\begin{equation*}
    c = \lim_{t\to\infty}\frac{w(x,t)}{t}.
\end{equation*}
Now this result gives hope for a positive result for the next question: Does there exist a function $w_\infty$ solving 
\begin{equation*}
    - \left( \Div \left( \frac{ Dw}{|Dw|} \right) +1 \right)|Dw| = \hat f(x) - c
\end{equation*}
such that 
\begin{equation*}
    \lim_{t\to\infty} w(x,t) -ct = w_\infty(x)
\end{equation*}
locally uniformly? Under further assumptions, the same authors proved in \cite{remarks} that this is the case in a radially symmetric setting as well as the case with no such symmetry, but with more technical assumptions. See Section $1.1$ in \cite{remarks} for specifics on those assumptions. For our model \eqref{cauchy}, the cutoff function would correspond to a natural situation where the crystal cannot dissolve once formed. We can perhaps think of the solution in which the crystal lives as being supersaturated. Thus, dissolution is not preferable by the system. Furthermore, the advection term would correspond to a flow in the solution transporting crystal deposits.

\subsection{Preliminaries and Main Results}

\subsubsection{Assumptions}
In the periodic setting we will always assume the following:
\begin{align}
    & \text{Let } g,f\in C(\T^n), \, f\geq0, \text{ and } \vec W\in \Lip(\T^n)^n;\tag{P1} \label{P1} \\
    &\cA: = \{x\in\T^n : f(x) = 0\}\subseteq \{x\in\T^n: \vec W(x) = \vec 0\};  \tag{P2} \label{P2}\\ 
    &\cA\neq \emptyset. \tag{P3} \label{P3}   
\end{align}
When imposing radial symmetry we assume the following:
\begin{align}
    & \text{Let } f(x) = F(|x|)\geq 0 \text{ with } F\in C^1([0,\infty)), \text{ and } g(x) = G(|x|) \tag{R1} \label{R1}\\
    & \qquad \text{ with } G\in C^1([0,\infty))\cap\Li([0,\infty)); \notag  \\
    & \lim_{r\to\infty}F(r) = c_F>0; \tag{R2} \label{R2} \\
    & \vec W\equiv \vec 0 ; \tag{R3} \\
    & \cA _{rad}: = \{r\in(0,\infty): F(r) = 0\} \neq \emptyset \tag{R4} \label{R4}
\end{align}

\subsubsection{Periodic Setting}

\begin{thm} \label{periodic largetime}
    Assume \eqref{P1}--\eqref{P3}. Let $u\in C(\T^n\times (0,\infty))$ be the unique viscosity solution to \eqref{cauchy}. Then, there exists a viscosity solution $v\in C(\T^n)$ to \eqref{periodic ergodic} such that 
    \begin{equation*}
        \lim_{t\to\infty} u(x,t) = v(x) \text{ uniformly in } \T^n.
    \end{equation*}
\end{thm}
\noindent Equation \eqref{periodic ergodic} is defined in Section $2$. 

\subsubsection{Radially Symmetric Setting}
 
\begin{thm} \label{lt behavior rad}
    Assume \eqref{R1}--\eqref{R4}. Let $U\in C((0,\infty)\times[0,\infty))$ be the unique viscosity solution to \eqref{cauchy}. Then, there exists a $V\in C((0,\infty))$ viscosity solution to \eqref{radial ergodic} such that 
    \begin{equation} \label{lg time behavior radial}
        \lim_{t\to\infty} U(r,t) = V(r) \quad \text{locally uniformly on } (0,\infty).
    \end{equation}
\end{thm}
\noindent Equation \eqref{radial ergodic} is defined in Section $3$.

The following result pertains to a representation formula for the limiting stationary solution:
\begin{thm} \label{Thm Rep formula for limit}
    Assume \eqref{R1}--\eqref{R4}. Let $U\in C\left( [0,\infty)\times (0,\infty) \right)$ be the unique viscosity solution to \eqref{radcauchy} and $V\in C([0,\infty))$ be the viscosity solution to \eqref{radial ergodic} such that \eqref{lg time behavior radial} holds. Then, we have the following representation formula for $V$:
    \begin{equation} \label{rep formula for limit}
        V(r) = \inf \left\{ 
        \inf_{\gamma\in\mathcal{C}(t,0;r,s),t>0} \left\{ 
\int_0^t F(\gamma(z))\,dz\right\} +\, v_G(s):s\in \cA_{rad}
        \right\},
    \end{equation}
    where $v_G(s): = \inf_{\gamma\in\mathcal{C}(t,0;s,\rho),\rho\in(0,\infty),t>0}\left\{ 
\int_0^t F(\gamma(z))\,dz+ G(\rho)\right\}$.
\end{thm}
\medskip

\section{Periodic setting}
It is convenient to rewrite \eqref{cauchy} into the following form: 

\begin{equation} \tag{C$^\prime$} \label{periodic cauchy}
\left\{
\begin{aligned}
    &u_t + \left(- a^{ij}(Du)(D^2u)_{ij} +|Du|\right)_+ +\vec W(x)\cdot Du=f(x) &&\text{ in } \T^n\times(0,\infty)\\
    &u(x,0) = g(x) &&\text{ on } \T^n,
    \end{aligned}
\right.
\end{equation}
with $a^{ij}(Du) := \left(\delta_{ij} - \frac{u_{x_i}u_{x_j}}{|Du|^2}\right)$ and cutoff $(\cdot)_+ := \max\{0,\cdot\}$. Here, $\delta_{ij}$ is the Kronecker delta. Moreover, $g\in C(\T^n)$ is a given initial data and $f\in C(\T^n)$ is a given non-negative source. The so-called \emph{wind flow} $\vec W$ is a Lipschitz vector field. Lastly, we define $\T^n:= \R^n/\Z^n$ as the usual $n$-dimensional flat torus. We will also consider the corresponding ergodic cell problem 
\medskip
\begin{equation} \tag{E} \label{periodic ergodic}
    \left(- a^{ij}(Dv)(D^2v)_{ij} +|Dv|\right)_+ +\vec W(x)\cdot Du = f(x) \quad \text{ in } \T^n. 
\end{equation}

In studying the large-time dynamics of solutions, it is useful to consider the Aubry set, which corresponds to dynamically invariant regions with respect to the Hamiltonian. Let us consider the two following sets:
\begin{equation*}
    \{x\in\T^n: f(x) = 0\}\quad \text{and} \quad \{x\in\T^n: \vec W(x) = \vec 0\}.
\end{equation*}
Because of \eqref{P2} and \eqref{P3}, 
\begin{equation*}
    \emptyset \neq \cA: = \{x\in\T^n : f(x) = 0\}\subseteq \{x\in\T^n: \vec W(x) = \vec 0\},
\end{equation*}
which is the corresponding Aubry set for our problem.
\begin{rem}
    The Aubry set $\cA$ is a monotonicity set for the Cauchy problem \eqref{periodic cauchy}. Indeed, assume that $u$ is smooth enough, and observe by the PDE that
\begin{align*}
    u_t &= f(x) - \left(- a^{ij}(Du)(D^2u)_{ij} +|Du|\right)_+ -\vec W(x)\cdot Du \\
    &\leq f(x) -\vec W(x)\cdot Du = 0,\notag
\end{align*}
for all $(x,t)\in\cA\times(0,\infty)$. Therefore, for $x\in\cA$ fixed, the map $t\mapsto u(x,t)$ is non-increasing. We will see that $\cA$ serves not only as a monotonicity set, but as a uniqueness set also. See Lemma \ref{comparison principle}.
\end{rem}

\begin{rem}
    The Aubry set $\cA$ contains all points $x\in\T^n$ such that when $f(x) = 0$, then necessarily $\vec W(x) = \vec 0$. Examples of this include $\vec W(x)\equiv\vec0$ or $\vec W(x) = \vec w f(x)$ with $\vec w$ being a constant vector field. 
\end{rem}

In order to prove Theorem \ref{periodic largetime}, we require a comparison principle on the Aubry set for the ergodic problem \eqref{periodic ergodic}. One can think of this as a Dirichlet problem with the Aubry set as the \emph{boundary} of the domain. 
\begin{lem} \label{comparison principle}
    Assume \eqref{P1}--\eqref{P3}. Let $v\in \USC(\T^n)\cap\Li(\T^n)$ and $w\in \LSC(\T^n)\cap\Li(\T^n)$ be a viscosity subsolution and supersolution, respectively, to \eqref{periodic ergodic}. Furthermore, let $v\leq w$ on $\cA$. Then, we necessarily have $v\leq w$ in $\T^n$.
\end{lem} 
\begin{proof}
Let $v,w$ be as stated in the hypothesis. Suppose for contradiction that there exists a point $x_0\in\T^n$ such that 
\begin{equation*}
    \max_{\T^n} \left(v(x)-w(x)\right) = v(x_0) - w(x_0) >0.
\end{equation*}
It is clear that $x_0\in\T^n\setminus \cA$. It is useful to consider the following rescaling of the viscosity supersolution $w$ to provide some room in the argument:
\begin{equation*}
w^\lambda(x):= \lambda w(x) \quad \text{ for } \lambda>1.
\end{equation*}
Observe that $w^\lambda$ is itself lower semi-continuous and bounded. We now make the following claim on the stability of maxima of upper semi-continuous functions:
We claim that there exists a $\lambda>1$ sufficiently close to one such that 
    \begin{equation*}
    \max_{\T^n} \left(v(x) - w^{\lambda}(x)\right) >0.
    \end{equation*}
    In order to prove the claim, we suppose for contradiction that we have a sequence $\{\lambda_k\}_{k\in\N}$ such that $\lambda_k\to1^+$ as $k\to\infty$, and we have that 
    \begin{equation*}
    \max_{\T^n} \left( v(x) - w^{\lambda_k}(x) \right) \leq 0 ,\quad \forall k\in\N.
    \end{equation*}
By upper semi-continuity of $v -w^{\lambda_k}$, we have that there exist a subsequence $\{\lambda_{k_j}\}_{j\geq1}$ and subsequence of points $\{x_j\}_{j\geq1}$ such that $x_j\to x_0$ as $j\to\infty$ and 
\begin{equation*}
\limsup_{j\to\infty} \left(v(x_j)-w^{\lambda_{k_j}}(x_j) \right)= v(x_0) -w(x_0).
\end{equation*}
Recall that $v(x_0) - w(x_0) >0.$ However, for every $j\geq1$ we necessarily have that 
\begin{equation*}
v(x_j)-w^{\lambda_{k_j}}(x_j) \leq 0.
\end{equation*}
This gives a contradiction. Hence, the claim is
true. 
Let us now fix such $\lambda>1$ satisfying the previous claim.
Now define the auxiliary function $\Psi^\varepsilon : \T^n\times\T^n\to \R$ by 
\begin{equation*}
    \Psi^\varepsilon (x,y) := v(x) - w^\lambda(y) - \frac{|x-y|^4}{\varepsilon}.
\end{equation*}
Observe that $\Psi^\varepsilon$ is upper semi-continuous, and so attains its maximum on a compact set. We denote the sequence of maximizers by $\{(x^\varepsilon,y^\varepsilon)\}_{\varepsilon>0}$. We follow with 3 more important claims:
We first claim that
    $\frac{|x^\varepsilon - y^\varepsilon|^4}{\varepsilon} \to 0$ as $\varepsilon \to 0$.
Denote $M_\varepsilon := \max_{\T^{2n}} \Psi^\varepsilon$. Note that $M_\varepsilon$ is decreasing with respect to $\varepsilon$. Moreover, 
\begin{equation*}
    \sup_{x=y}\Psi^\varepsilon(x,y) \leq M_\varepsilon \leq M_1.
\end{equation*}
Therefore, $M_0 := \lim_{\varepsilon \to 0} M_\varepsilon$ is well-defined and finite. It follows then that 
\begin{align*}
    M_{2\varepsilon} - \left(\frac{1}{\varepsilon} - \frac{1}{2\varepsilon}\right)|x^\varepsilon - y^\varepsilon|^4 &\geq v(x^\varepsilon) - \lambda w(\frac{y^\varepsilon}{\lambda}) - \frac{|x^\varepsilon - y^\varepsilon|^4}{2\varepsilon} - \left(\frac{1}{\varepsilon} - \frac{1}{2\varepsilon}\right)|x^\varepsilon - y^\varepsilon|^4 \\
    &= v(x^\varepsilon) - \lambda w(\frac{y^\varepsilon}{\lambda}) - \frac{|x^\varepsilon - y^\varepsilon|^4}{\varepsilon} = M_\varepsilon.
\end{align*}
Hence, we have 
\begin{equation*}
    M_{2\varepsilon} - M_\varepsilon \geq \left(\frac{1}{\varepsilon} - \frac{1}{2\varepsilon}\right)|x^\varepsilon - y^\varepsilon|^4 = \frac{|x^\varepsilon - y^\varepsilon|^4}{\varepsilon},
\end{equation*}
and sending $\varepsilon \to 0$ we have proven the claim. 

We next claim that 
    $(x^\varepsilon,y^\varepsilon) \to (\hat x,\hat y)\in \T^n \times \T^n$ up to a subsequence. We are in a compact setting, and so the sequence $\{(x^\varepsilon,y^\varepsilon)\}_\varepsilon$ has a convergent subsequence $\{(x^{\varepsilon_j},y^{\varepsilon_j})\}_j$. 
Lastly, we claim that 
    necessarily $\hat x = \hat y$.
Observe that by the previous claim, we have that $|x^\varepsilon - y^\varepsilon|^4$ is $o(\varepsilon)$. So, along the convergent subsequence we have that the limit points are the same. Furthermore, we necessarily have that $\hat x\in \T^n\setminus \cA$. We will use this important fact later. Note that whenever we send $\varepsilon\to 0$, we mean along this particular convergent subsequence. 

We continue by applying the Crandall-Ishii Lemma or Theorem of Sums, which can be found in \cite{user}. First denote the test function in the auxiliary function $\Psi^\varepsilon$ by $\varphi(x,y):= \frac{|x-y|^4}{\varepsilon}$. Then, by the Crandall-Ishii Lemma we have for $\kappa>0$ sufficiently small such that $2\kappa D^2 \varphi(x^\varepsilon,y^\varepsilon) < I_{2n}$, there exist symmetric $n\times n$ matrices $\X^\varepsilon, \Y^\varepsilon$ satisfying 
\begin{equation} \label{super jet} \tag{Super Jet}
    \left(v(x^\varepsilon),D_x\varphi(x^\varepsilon,y^\varepsilon), \X^\varepsilon \right) \in \overline {\begin{bf}J\end{bf}^{2,+}}[v(x^\varepsilon)]
\end{equation}
and 
\begin{equation} \label{sub jet} \tag{Sub Jet}
    \left(w^\lambda(y^\varepsilon),-D_y\varphi(x^\varepsilon,y^\varepsilon), \Y^\varepsilon \right) \in \overline {\begin{bf}J\end{bf}^{2,-}}\left[w^\lambda(y^\varepsilon)\right].
\end{equation}
Moreover, we have the key inequality 
\begin{equation} \label{key ineq}
    \begin{pmatrix}
        \X^\varepsilon & 0\\
        0 & -\Y^\varepsilon
    \end{pmatrix} 
    \leq D^2\varphi(x^\varepsilon,y^\varepsilon) + 2\kappa
\, (D^2\varphi(x^\varepsilon,y^\varepsilon))^2.
\end{equation}
Denote
\[
p^\varepsilon:= D_x\varphi(x^\varepsilon,y^\varepsilon) = -D_y\varphi(x^\varepsilon,y^\varepsilon) = \frac{4|x^\varepsilon - y^\varepsilon|^2}{\varepsilon}(x^\varepsilon - y^\varepsilon).
\]
We will use the fact that $w$ is a viscosity supersolution and the $1$-homogeneity of \eqref{periodic ergodic} to deduce which PDE $w^\lambda$ is a viscosity supersolution of. Observe if $w$ were smooth, we would have
\begin{equation*}
\left\{
\begin{aligned}
&Dw^\lambda(y) = D[\lambda w(y)] = \lambda Dw(y)\\
&D^2w^\lambda(y) = D^2[\lambda w(y)] = \lambda D^2w(y)\\
&-a^{ij}(Dw^\lambda(y)) = -a^{ij}(D[\lambda w(y)]) = -a^{ij}(Dw(y)).
\end{aligned}
\right.
\end{equation*}
So, for $y\in\T^n\setminus \cA$ we have  
\begin{equation*}
    \left( -a^{ij}\left(Dw(y)\right)\left(D^2w(y)\right)_{ij} + \left|Dw(y)\right|
    \right)_+ + \vec W(y)\cdot Dw(y)
    \geq f(y).
\end{equation*}
Using the previous identities give
\begin{equation*}
    \left( 
    -a^{ij}\left(Dw^\lambda(y)\right)\left(\frac{1}{\lambda} D^2w^\lambda(y)\right)_{ij} + \left|\frac{1}{\lambda}Dw^\lambda(y)\right|
    \right)_+ + \vec W(y)\cdot \frac{1}{\lambda}Dw^\lambda(y)
    \geq f\left(y\right).
\end{equation*}
Which is equivalent to 
\begin{equation}\label{new super}
   \left(
   -a^{ij}\left(Dw^\lambda(y)\right)\left(D^2w^\lambda(y)\right)_{ij} + \left|Dw^\lambda(y)\right|
   \right)_+ + \vec W(y)\cdot Dw^\lambda(y)
   \geq \lambda f(y).
\end{equation}
Hence, $w^\lambda$ is a viscosity solution to \eqref{new super} for all $y\notin\cA$. We now make another claim, that for $\varepsilon>0$ sufficiently small we necessarily have that $x^\varepsilon\neq y^\varepsilon$. Suppose, for contradiction, that $x^\varepsilon = y^\varepsilon$ with $x^\varepsilon$ being arbitrarily close to $\hat x\in\T^n\setminus \cA$. Then, the map $y\mapsto w^\lambda(y) + \frac{|x^\varepsilon - y|^4}{\varepsilon}$ has a min at $x^\varepsilon$, and so by the supersolution test we have 
\begin{equation*}
    \left(
    -\left(\delta_{ij} - \eta\otimes\eta\right)\left(\begin{bf}0\end{bf}_n\right)_{ij} + |\vec 0|
    \right)_+ + \vec W(y)\cdot \vec 0
    \geq \lambda f(x^\varepsilon)>0,
\end{equation*}
where $\eta\in\R^n$ with $|\eta|\leq1$ and $\begin{bf}0\end{bf}_n$ is the zero $n\times n$ matrix. This gives a contradiction. Therefore, we must have $x^\varepsilon\neq y^\varepsilon$ for arbitrarily small $\varepsilon$.
We now apply the subsolution and supersolution tests using \eqref{super jet} and \eqref{sub jet}, which give
\begin{equation} \tag{subtest}
    \left(
    -a^{ij}(\begin{bf} p\end{bf}^\varepsilon) \X^\varepsilon_{ij} + |\begin{bf} p\end{bf}^\varepsilon|
    \right)_+ + \vec W(x^\varepsilon)\cdot \begin{bf} p\end{bf}^\varepsilon
    \leq f(x^\varepsilon)
\end{equation}
and 
\begin{equation} \tag {supertest}
    \left(
    -a^{ij}(\begin{bf} p\end{bf}^\varepsilon) \Y^\varepsilon_{ij} + |\begin{bf} p\end{bf}^\varepsilon|
    \right)_+  + \vec W(y^\varepsilon)\cdot \begin{bf} p\end{bf}^\varepsilon
    \geq \lambda f(y^\varepsilon).
\end{equation}
We combine the two inequalities from above giving 
\begin{align} \label{periodic combo}
    \left(
-a^{ij}(\begin{bf} p\end{bf}^\varepsilon) \X^\varepsilon_{ij} + |\begin{bf} p\end{bf}^\varepsilon|
    \right)_+ 
    &-
\left(
    -a^{ij}(\begin{bf} p\end{bf}^\varepsilon) \Y^\varepsilon_{ij} + |\begin{bf} p\end{bf}^\varepsilon|
    \right)_+ \\
    &\leq 
    f(x^\varepsilon) - \lambda f(y^\varepsilon) + \left[\vec W(y^\varepsilon) - \vec W(x^\varepsilon) \right] \cdot \begin{bf} p\end{bf}^\varepsilon. \notag
\end{align}
We consider first the right hand side of \eqref{periodic combo}. To that end, observe that
\begin{align} \label{rhs}
    f(x^\varepsilon) - \lambda f(y^\varepsilon) &+ \left[\vec W(y^\varepsilon) - \vec W(x^\varepsilon) \right] \cdot \begin{bf} p\end{bf}^\varepsilon \notag \\
    &\leq 
    f(x^\varepsilon) - \lambda f(y^\varepsilon) + \left|\vec W(y^\varepsilon) - \vec W(x^\varepsilon) \right| |\begin{bf} p\end{bf}^\varepsilon| \notag \\
    &\leq f(x^\varepsilon) - \lambda f(y^\varepsilon) + L_{\vec W} |x^\varepsilon - y^\varepsilon\|\begin{bf} p\end{bf}^\varepsilon| \notag \\
    & = f(x^\varepsilon) - \lambda f(y^\varepsilon) + 4L_{\vec W} \frac{|x^\varepsilon - y^\varepsilon|^4}{\varepsilon}, 
\end{align} 
where we used the fact that $\vec W$ is a Lipschitz vector field with Lipschitz constant $L_{\vec W}$. 
Now, we deal with the left hand side of \eqref{periodic combo}. In the case where 
\begin{equation*}
\left(
    -a^{ij}(\begin{bf} p\end{bf}^\varepsilon) \Y^\varepsilon_{ij} + |\begin{bf} p\end{bf}^\varepsilon|
    \right)_+ = 0, 
\end{equation*}
we clearly have non-negativity of the left hand side. This with \eqref{rhs} gives the following inequality:
\begin{equation} \label{almost contrad}
    0\leq f(x^\varepsilon) - \lambda f(y^\varepsilon) + 4L_{\vec W} \frac{|x^\varepsilon - y^\varepsilon|^4}{\varepsilon}
\end{equation}
We now consider when 
\begin{equation*}
\left(
    -a^{ij}(\begin{bf} p\end{bf}^\varepsilon) \Y^\varepsilon_{ij} + |\begin{bf} p\end{bf}^\varepsilon|
    \right)_+ =
    -a^{ij}(\begin{bf} p\end{bf}^\varepsilon) \Y^\varepsilon_{ij} + |\begin{bf} p\end{bf}^\varepsilon|. 
\end{equation*}
Then, we have the basic estimate 
\begin{align*}
\left(
-a^{ij}(\begin{bf} p\end{bf}^\varepsilon) \X^\varepsilon_{ij} + |\begin{bf} p\end{bf}^\varepsilon|
    \right)_+ 
    &-
\left(
    -a^{ij}(\begin{bf} p\end{bf}^\varepsilon) \Y^\varepsilon_{ij} + |\begin{bf} p\end{bf}^\varepsilon|
    \right)_+  \notag \\
    \geq 
    -a^{ij}(\begin{bf} p\end{bf}^\varepsilon) \X^\varepsilon_{ij} + |\begin{bf} p\end{bf}^\varepsilon|
    &+a^{ij}(\begin{bf} p\end{bf}^\varepsilon) \Y^\varepsilon_{ij} - |\begin{bf} p\end{bf}^\varepsilon| \notag \\
    &=  -a^{ij}(\begin{bf} p\end{bf}^\varepsilon) [\X^\varepsilon - \Y^\varepsilon]_{ij}.
\end{align*}
It follows that 
\begin{equation}
    -a^{ij}(\begin{bf} p\end{bf}^\varepsilon) [\X^\varepsilon - \Y^\varepsilon]_{ij} \leq f(x^\varepsilon) - \lambda f(y^\varepsilon) + 4L_{\vec W} \frac{|x^\varepsilon - y^\varepsilon|^4}{\varepsilon}.
\end{equation}
By the key inequality \eqref{key ineq} in the Theorem of Sums, we deduce that $\X^\varepsilon\leq \Y^\varepsilon$. Indeed, 
\[
\begin{pmatrix}
        \X^\varepsilon & 0\\
        0 & -\Y^\varepsilon
    \end{pmatrix} 
    \leq D^2\varphi(x^\varepsilon,y^\varepsilon) + O(\kappa), 
\]
where the Hessian term annihilates vectors of the form $\begin{bmatrix} \xi & \xi \end{bmatrix}^T$ in the quadratic form, and when we take the trace of the inequality we have the desired result.
This along with the fact that $a^{ij}$ is non-negative definite, we have that left hand side is the trace of a non-positive definite matrix. 
Therefore,
\[
-a^{ij}(\begin{bf} p\end{bf}^\varepsilon) [\X^\varepsilon - \Y^\varepsilon]_{ij}\geq 0,
\]
for any $\varepsilon$. Hence, we arrive at \eqref{almost contrad}, that is, 
\[
0\leq f(x^\varepsilon) - \lambda f(y^\varepsilon) + 4L_{\vec W} \frac{|x^\varepsilon - y^\varepsilon|^4}{\varepsilon},
\]
and by sending $\varepsilon\to 0$ up to a subsequence we have
\begin{equation*}
0\leq f(\hat x) - \lambda f(\hat x) = (1-\lambda)f(\hat x)<0.
\end{equation*}
This is a clear contradiction, as $\lambda>1$ and $\hat x\in\T^n\setminus \cA$. Thus, we have proven the Comparison Principle for \eqref{periodic ergodic}.
\end{proof}
\medskip
%\begin{rem}
%    In the argument above, if we were to rescale the subsolution instead, we would require $\lambda<1$ and have the strict inequality reversed in the definition of $F_1$. 
%\end{rem}
How do we connect a solution of the Cauchy problem to a solution of the ergodic problem? It is standard to consider the half-relaxed limits of a solution to the former. We recall the definitions here. 
    Let $u\in C(\T^n\times(0,\infty))$ be the viscosity solution to \eqref{periodic cauchy}. The half-relaxed limits $u^+\in\USC(\T^n)$ and $u^-\in \LSC(\T^n)$ are defined as
    \[
    u^+(x) := \limsups\limits_{t\to\infty}{u(x,t)} = \lim_{t\to\infty} \sup\left\{u(y,s):|x-y|\leq\frac{1}{s}, s\geq t\right\}
    \]
    and 
    \[
    u^-(x) := \liminfs\limits_{t\to\infty}{u(x,t)} = \lim_{t\to\infty} \inf\left\{u(y,s):|x-y|\leq\frac{1}{s}, s\geq t\right\}.
    \]
The idea is to apply Lemma \ref{comparison principle} to the half-relaxed limits of $u$, which will in turn prove that the large-time limit exists on all of $\T^n$, as desired. The following proposition shows why this is possible. 
\begin{prop} \label{relaxed limit}
    Let $u^+,u^-$ be the half-relaxed limits of $u$, the unique viscosity solution to \eqref{periodic cauchy}, described above. Then, $u^+, u^-$ are a viscosity subsolution and supersolution to \eqref{periodic ergodic}, respectively. 
\end{prop}
%\begin{proof}
%    Prove Proposition \ref{relaxed limit} here. 
%\end{proof}
In fact, this is known as the stability of viscosity solutions, which makes the theory so successful. We refer the reader to \cite{user} and \cite{dynamical} (Theorem $7.21$, Section $7.5$) for more information on the stability of viscosity solutions.
We now have the machinery to prove Theorem \ref{periodic largetime}. 
\begin{proof}[Proof of Theorem \ref{periodic largetime}]
    Recall from Proposition \ref{relaxed limit} that $u^+$ and $u^-$ are a viscosity subsolution and supersolution to \eqref{periodic ergodic}, respectively. Observe also that by the monotonicity property of the Aubry set $\cA$, we have that $u^+=u^-$ on $\cA$. Therefore, by Lemma \ref{comparison principle} we conclude that indeed $u^+\leq u^-$ in $\T^n$. By definition, $u^-\leq u^+$. Hence, $u^-=u^+$. 
\end{proof}

%%%%%%%%%%%%%%%%%%%%%%%%%%%%%%%%%%%%%%%%%%%%%%%%%%%%%%%%%%%%%%%%%%
\section{Radially symmetric setting}

In this section, we always assume $f(x) = F(|x|)$, $g(x) = G(|x|)$ for $F,G\in C^1([0,\infty))$. Together with $\vec W\equiv \vec 0$, the Cauchy problem \eqref{cauchy} becomes    
\begin{equation}\label{radcauchy} \tag{C$_{rad}$}
    \left\{
\begin{aligned}
    &U_t + \left(- \frac{(n-1)}{r}U_r + |U_r|\right)_+ = F(r) &&\text{ in } (0,\infty)\times(0,\infty),\\
    &U(r,0) = G(r) &&\text{ on } (0,\infty).
\end{aligned}
\right.
\end{equation}
Observe that this is a singular first-order Hamilton-Jacobi equation with Hamiltonian $H: (0,\infty)\times\R\to \R$ defined by 
\begin{equation*}
    H(r,p): = \left( -\frac{n-1}{r}p +|p|\right)_+ -F(r).
\end{equation*}
We remark that this Hamiltonian is convex in $p$, and so by standard Hamilton-Jacobi theory we have the following optimal control representation formula for the unique viscosity solution to \eqref{radcauchy}:
\begin{equation} \label{rep formula}
    U(r,t) = \inf\left\{ \int_0^t L(\gamma(z),\dot\gamma(z)) \,dz + G(\gamma(0)) : \gamma\in \AC \text{ with } \gamma(t) = r \right\},
\end{equation}
where $L:(0,\infty)\times\R\to\R$ is the Lagrangian of the system. For more information on this representation formula, we refer the reader to Chapter 2 of \cite{tran}. Recall for convex Hamiltonians, the Lagrangian is given by the Legendre transform. That is, 
\begin{equation*}
    L(r,v) = \sup_{p\in\R} \left\{ pv - H(r,p)\right\} = \sup_{p\in\R} \left\{ pv - \left( -\frac{(n-1)}{r}p +|p|\right)_+ + F(r) \right \}.
\end{equation*}
We will now compute the Lagrangian explicitly. We break the domain $(0,\infty)$ into two regions: $(0,n-1]$ and $(n-1,\infty)$. 

For $\gamma(z)\in(0,n-1]$, the Lagrangian $L$ reduces to 
\begin{equation*}
    L(\gamma(z),\dot \gamma(z)) = \left\{
    \begin{aligned}
        &+\infty \quad \text{if}\quad \dot\gamma(z)<-1-\frac{(n-1)}{\gamma(z)},\\
        &F(\gamma(z)) \quad \text{if} \quad -1-\frac{(n-1)}{\gamma(z)}\leq \dot\gamma(z)\leq 0, \\
        &+\infty \quad \text{if} \quad \dot\gamma(z)>0.
    \end{aligned}
    \right.
\end{equation*}
Now for $\gamma(z)\in(n-1,\infty)$, the Lagrangian reduces to 
\begin{equation*}
    L(\gamma(z),\dot\gamma(z)) = \left\{
    \begin{aligned}
        &+\infty \quad \text{if} \quad \dot\gamma(z)<-1 - \frac{(n-1)}{\gamma(z)},\\
        &F(\gamma(z)) \quad \text{if} \quad -1 - \frac{(n-1)}{\gamma(z)}\leq \dot\gamma(z)\leq 1 - \frac{(n-1)}{\gamma(z)},\\
        &+\infty \quad \text{if} \quad \dot\gamma(z)>1 - \frac{(n-1)}{\gamma(z)}.
    \end{aligned}
    \right.
\end{equation*}
Of course, as the value function is minimizing this action, it will avoid any trajectories that make the Lagrangian infinite. Therefore, we can exclude these inadmissible trajectories to write a more useful representation formula. 

%With an explicit Lagrangian, we can extract more information on the behavior of $U(r,t)$. Let us discuss the finite speed of propagation of this first-order Hamilton-Jacobi equation and exclude inadmissible trajectories. 

Note that in the region $(0,n-1]$, in order for the Lagrangian to be finite, the velocity $\dot\gamma(z)$ is only able to be non-positive. Geometrically, if we fix $r^*\in (0,n-1]$ and consider the value function $U(r^*,t)$, then this implies that admissible trajectories can only be drawn from the right of $\gamma(t) = r^*$. See Figure \ref{fig:fromright}.

%%%%%%%%%%%%%%%%BEGIN FIGURE 1%%%%%%%%%%%%%%%%%%%%%%%%%%%%%%%%%%%%%%%%%%%%%%%%%%%%
\begin{figure}[htp]
\centering
\caption{Trajectories from the right.}
\label{fig:fromright}
\begin{tikzpicture}[
    >=Stealth,
    every path/.style={thick},
    dot/.style={circle, fill=black, inner sep=1.5pt},
    decoration={
        markings,
        mark=at position 0.7 with {\arrow{>}} % Arrow at 70% of path
    }
]

% Define the target point r and terminal time t
\coordinate (r) at (2,4); % (x,t) coordinates
\node[dot, label=above:{$(r^*,t)$}] at (r) {};

% Axes: r (spatial) horizontal, s (time) vertical
\draw[->] (0,0) -- (12,0) node[right]{$r$}; % x-axis
\draw[->] (0,0) -- (0,5) node[above]{time $z$};  % t-axis

% Terminal time line
\draw[dashed] (0,4) -- (2,4) ;
\node[left] at (0,4) {$t$};
\draw[dashed] (2,0) -- (2,4);
%\filldraw[black] (2,4) circle (2pt);

% Draw three trajectories approaching r from the right (x > x_r)
% Trajectory 1 - shallow approach
\draw[postaction={decorate}, black] 
    (4,0) to[out=100, in=350] (r);

% Trajectory 2 - steeper approach
\draw[postaction={decorate}, black] 
    (4.5,0) to[out=95, in= 280] (r);

% Trajectory 3 - direct approach
\draw[postaction={decorate}, black] 
    (7.5,0) to[out=120, in=0] (3.3,3.5) to[out=180, in=275] (r);

% Trajectory 4 - direct approach
%\draw[postaction={decorate}, black] 
%    (5,0) to[out=90, in=0] (3.3,3.5) to[out=180, in=270] (r);

% Trajectory 5 - direct approach
\draw[postaction={decorate}, black] 
        (2.5,0) to[out=93, in= 270]  (r);

% Trajectory 6 - direct approach
\draw[postaction={decorate}, black] 
    (3.5,0) to[out=90, in=345] (r);

% Trajectory 7 - direct approach
\draw[postaction={decorate}, black] 
    (5,0) to[out=90, in=0] (r);

% Trajectory 8 - direct approach
\draw[postaction={decorate}, black] 
    (9,2) to[out=134, in=0] (r);

% Trajectory 9 - direct approach
\draw[postaction={decorate}, black] 
    (9,1.2) to[out=160, in=0] (r);

% Labels
%\node at (4.2,1.5) [red] {$\mathbf{\gamma}_1$};
%\node at (4.7,1) [blue] {$\mathbf{\gamma}_2$};
%\node at (7.2,2) [green!70!black] {$\mathbf{\gamma}_3$};

% Spatial reference points
%\foreach \x in {} 
%    \draw (\x,0.1) -- (\x,-0.1) node[below] {$x_{\x}$};

% r^* line
\draw (2,0.1) -- (2,-0.1) node[below] {$r^*$};

% n-1 line
\draw (9,0.1) -- (9,-0.1) node[below] {$n-1$};
\draw[dotted] (9,0) -- (9,5);

% 0 line
\draw (0,0.1) -- (0,-0.1) node[below] {$0$};

% Highlight that all approaches are from x > x_r
%\draw[dotted] (2,0) -- (2,4);
%\node[below] at (3,-0.1) {$x_r$};
%\node[right] at (3.5,3.5) {All trajectories\\approach from $x > x_r$};

\end{tikzpicture}
\end{figure}
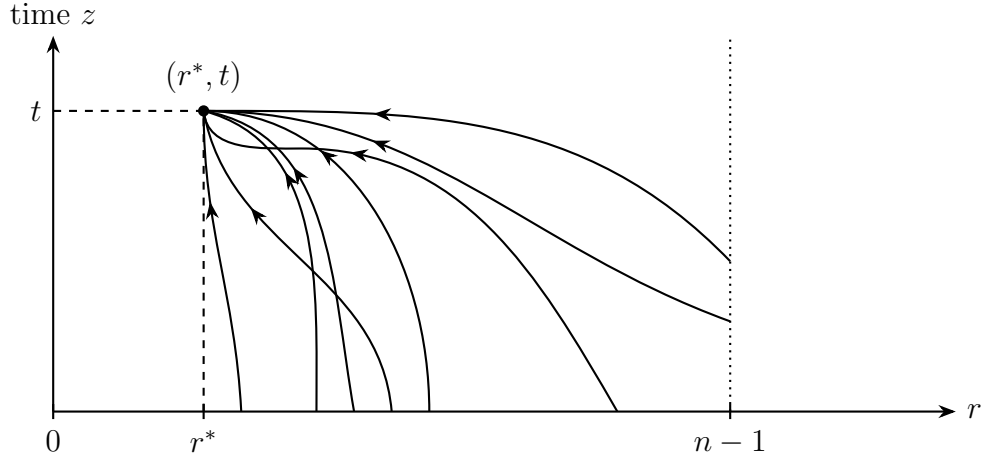
%%%%%%%%%%%%%%%%%%END FIGURE 1%%%%%%%%%%%%%%%%%%%%%%%%%%%%%%%%%%%%%%%%%%%%%%%%%%%%

In $(n-1,\infty)$, the allowed velocities can be zero, positive, and negative. So, for $r^*\in(n-1,\infty)$ and value function $U(r^*,t)$, trajectories can be drawn from the left and right of $\gamma(t)= r^*$. However, trajectories starting in or crossing into $(0,n-1]$ are inadmissible for reasons above. 
    
Next, we aim to further eliminate inadmissible trajectories by studying local \emph{speed limits}. For $\gamma(z)\in(0,n-1]$, we have 
\begin{equation*}
    -1 - \frac{(n-1)}{\gamma(z)} \leq \dot\gamma(z) \leq 0
\end{equation*}
for allowed velocities. As $\gamma(z)\to0$, the lower bound tends to $-\infty$.  This means as the trajectory nears the origin, it may speed up. On the other hand, we see that for $\gamma(z)\in(n-1,\infty)$, we have that the allowed velocities are
\begin{equation*}
    -1 - \frac{(n-1)}{\gamma(z)} \leq \dot\gamma(z) \leq 1-\frac{(n-1)}{\gamma(z)}.
\end{equation*}
Furthermore, if $\dot\gamma(z)<-2$ or $\dot\gamma(z)>1$, then the Lagrangian is guaranteed to become infinite. This \emph{speed limit} provides a \emph{cone of influence}. In fact, let $r^*\in(n-1,\infty)$, and consider the value function $U(r^*,t)$. Observe that a trajectory $\gamma_1(z)$ with $\gamma_1(0) = r^* - t$ and $\dot\gamma_1\equiv 1$ will reach $r^*$ at time $t$. Hence, if $r^* - t> n-1$, then any trajectory $\gamma$ beginning at radius $r\leq r^* - t$ is inadmissible. If $r^*- t\leq n-1$, then any trajectory $\gamma$ with initial condition in $(n-1,r^*]$ will be admissible. Now consider a trajectory $\gamma_2(z)$ with $\gamma_2(0) = r^*+2t$ and $\dot \gamma_2\equiv-2$. We will have that $\gamma_2(t) = r^*$. Therefore, for any trajectory $\gamma$ with $\gamma(0)\geq r^*+2t $ is inadmissible. We refer the reader to Figure \ref{fig:cone fig} to see the \emph{cone of influence} in $(n-1,\infty)$. 

%%%%%%%%%%%%%%%%BEGIN FIGURE 2%%%%%%%%%%%%%%%%%%%%%%%%%%%%%%%%%%%%%%%%%%%%%%%%%%%%
\begin{figure}[ht]
\centering
\caption{Cone of admissible trajectories.}
\label{fig:cone fig}
\begin{tikzpicture}[
    >=Stealth,
    every path/.style={thick},
    dot/.style={circle, fill=black, inner sep=1.5pt},
    decoration={
        markings,
        mark=at position 0.7 with {\arrow{>}} % Arrow at 70% of path
    }
]

% Define the target point r and terminal time t
\coordinate (r) at (5,4); % (x,t) coordinates
\node[dot, label=above:{$(r^*,t)$}] at (r) {};

% Axes: r (spatial) horizontal, s (time) vertical
\draw[->] (0,0) -- (12,0) node[right]{$r$}; % x-axis
\draw[->] (0,0) -- (0,5) node[above]{time $z$};  % t-axis

% Terminal time line
\draw[dashed] (0,4) -- (5,4) ;
\node[left] at (0,4) {$t$};
\draw[dashed] (5,0) -- (5,4);
%\filldraw[black] (5,4) circle (2pt);

% Draw three trajectories approaching r from the right (x > x_r)
% Trajectory 1 - shallow approach
\draw[postaction={decorate}, black] 
    (4,0) to[out=100, in=275] (r);

% Trajectory 2 - steeper approach
\draw[postaction={decorate}, black] 
    (8,0) to[out=95, in= 280] (r);

% Trajectory 3 - direct approach
\draw[postaction={decorate}, black] 
    (7.5,0) to[out=120, in=260] (r);

% Trajectory 4 - direct approach
%\draw[postaction={decorate}, black] 
%    (5,0) to[out=90, in=0] (3.3,3.5) to[out=180, in=270] (r);

% Trajectory 5 - direct approach
\draw[postaction={decorate}, black] 
        (3,0) to[out=93, in= 270]  (r);

% Trajectory 6 - direct approach
\draw[postaction={decorate}, black] 
    (3.5,0) to[out=90, in=325] (r);

% Trajectory 7 - direct approach
\draw[postaction={decorate}, black] 
    (6,0) to[out=90, in=240] (r);

% Trajectory 8 - direct approach
\draw[postaction={decorate}, black] 
    (9,0) to[out=134, in=300] (r);

% Trajectory 9 - direct approach
\draw[postaction={decorate}, black] 
    (9.5,0) to[out=160, in=315] (r);

% Labels
%\node at (4.2,1.5) [red] {$\mathbf{\gamma}_1$};
%\node at (4.7,1) [blue] {$\mathbf{\gamma}_2$};
%\node at (7.2,2) [green!70!black] {$\mathbf{\gamma}_3$};

% Spatial reference points
%\foreach \x in {} 
%    \draw (\x,0.1) -- (\x,-0.1) node[below] {$x_{\x}$};

% r^* line
\draw (5,0.1) -- (5,-0.1) node[below] {$r^*$};

% r^* - t line
\draw (2.5,0.1) -- (2.5,-0.1) node[below] {$r^* - t$};
\draw[dashed,blue] (2.5,0) -- (5,4);

% r^* + 2t line
\draw (10,0.1) -- (10,-0.1) node[below] {$r^* + 2 t$};
\draw[dashed,blue] (10,0) -- (5,4);

% n-1 line
\draw (1,0.1) -- (1,-0.1) node[below] {$n-1$};
\draw[dotted] (1,0) -- (1,5);

% 0 line
\draw (0,0.1) -- (0,-0.1) node[below] {$0$};

% Highlight that all approaches are from x > x_r
%\draw[dotted] (2,0) -- (2,4);
%\node[below] at (3,-0.1) {$x_r$};
%\node[right] at (3.5,3.5) {All trajectories\\approach from $x > x_r$};

\end{tikzpicture}
\end{figure}
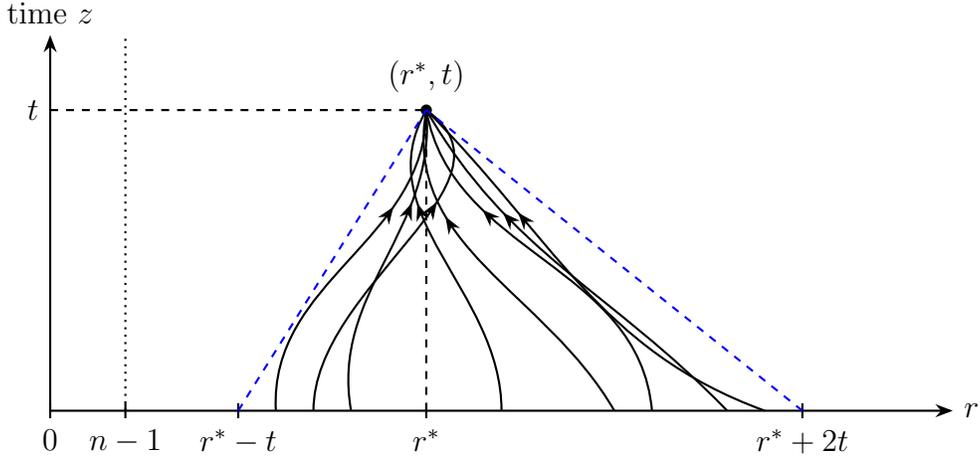
%%%%%%%%%%%%%%%%%%END FIGURE 2%%%%%%%%%%%%%%%%%%%%%%%%%%%%%%%%%%%%%%%%%%%%%%%%%%%%

We use similar notation from \cite{remarks} to denote the class of admissible trajectories to this optimal control problem. Indeed, set 
\begin{align} \label{admissible class}
    \mathcal{C}(t,0;r,s): &= \bigg\{
    \gamma\in AC([0,t];(0,\infty)): \gamma(t) = r, \gamma(0) = s, \text{ for a.e. } z\in(0,t) \\
    & -1 - \frac{(n-1)}{\gamma(z)} \leq \dot \gamma(z)\leq \left(1 - \frac{(n-1)}{\gamma(z)} \right)\chi_{(n-1,\infty)}(\gamma(z))
    \bigg\}. \notag
\end{align}
Thus, we can refine \eqref{rep formula} into 
\begin{align} \label{refined rep formula}
    U(r,t) = \inf_{\gamma\in\mathcal{C}(t,0;r,s),s\in(0,\infty)}\left\{ 
\int_0^t F(\gamma(z))\,dz+ G(\gamma(0))\right\}.
\end{align}

As in the periodic setting, we consider the corresponding ergodic problem and study the behavior of solutions. Under radial symmetry, the ergodic problem is 
\begin{equation} \label{radial ergodic} \tag{E$_{rad}$}
    \left(-\frac{(n-1)}{r}V_r + |V_r| \right)_+ = F(r) \quad \text{ in } (0,\infty).
\end{equation}
Moreover, we have a similar definition for the Aubry set, i.e., 
\begin{equation*}
     \cA_{rad}: = \{r\in(0,\infty):F(r) = 0\}.
\end{equation*}
By \eqref{R4},
\begin{equation*}
    \cA _{rad} \neq \emptyset.
\end{equation*}
\begin{rem}
    We have that $\cA_{rad}$ serves as a monotonicity set for $U$. Indeed, by using the representation formula \eqref{refined rep formula} we have for fixed $r^*\in\cA_{rad}$ and $t_1\leq t_2 $ with $t_1,t_2\in(0,\infty)$ that
    \begin{equation} \label{monotonicity}
        U(r^*,t_2) \leq U(r^*,t_1).
    \end{equation}
To see \eqref{monotonicity}, we consider the value function at $(r^*,t_2)$, i.e., 
\begin{align*}
    U(r^*,t_2) &= \inf_{\gamma\in\mathcal{C}(t,0;r^*,s),s\in(0,\infty)}\left\{ 
\int_0^{t_2} F(\gamma(z))\,dz+ G(\gamma(0))\right\} \\
    & = \inf_{\gamma\in\mathcal{C}(t,0;r^*,s),s\in(0,\infty)} \left\{ 
    \int_0^{t_1} F(\gamma(z))\,dz + \int_{t_1}^{t_2} F(\gamma(z))\,dz + G(\gamma(0))\right\}.
\end{align*}
In order to relate this to the value $U(r^*,t_1)$, we restrict the class of admissible trajectories by only considering ones that travel to $r^*$ in time $t_1$ and remain there until time $t_2$. This gives 
\begin{align*}
    U(r^*,t_2)&\leq \inf_{\gamma\in\mathcal{C}(t,0;r^*,s),s\in(0,\infty)} \bigg\{ 
    \int_0^{t_1} F(\gamma(z))\,dz + \int_{t_1}^{t_2} F(\gamma(z))\,dz + G(\gamma(0))\bigg|\\
    & \qquad \qquad \gamma(z)=r^* \quad  \forall \, t_1\leq z\leq t_2
    \bigg\} \\
    & = \inf_{\gamma\in\mathcal{C}(t,0;r^*,s),s\in(0,\infty)} \left\{ 
    \int_0^{t_1} F(\gamma(z))\,dz + \int_{t_1}^{t_2} F(r^*)\,dz + G(\gamma(0))\right\} \\
    & = \inf_{\gamma\in\mathcal{C}(t,0;r^*,s),s\in(0,\infty)} \left\{ 
    \int_0^{t_1} F(\gamma(z))\,dz + \int_{t_1}^{t_2} 0\,dz + G(\gamma(0))\right\} \\
    & = U(r^*,t_1).
\end{align*}
It is important that $r^*\in\cA_{rad}$, so that the strategy of waiting at the terminal point does not add to the running cost. Therefore, we see that the map $t\mapsto U(r,t)$ for $r\in\cA_{rad}$ is non-increasing.
\end{rem}

Again, we require the Comparison Principle for the ergodic problem in order to establish the large-time limit, i.e., Theorem \ref{lt behavior rad}. As the behavior of the viscosity subsolutions and supersolutions vary with respect to the radius, we will divide and conquer by establishing the Comparison Principle for different regions of $\R^+$. To that end, we define the following important radii:
\begin{align*} 
    &S_0:= \min\{r\leq n-1: r\in\cA_{rad}\}, \quad S_1:= \max\{r\leq n-1:r\in\cA_{rad}\}, \\
    &R_0:= \min\{r\geq n-1: r\in\cA_{rad}\}, \quad R_1:= \max\{r\geq n-1:r\in\cA_{rad}\}. 
\end{align*}
%\begin{lem}[Comparison Principle on $(0,n-1)$] \label{1st radial CP}
%    Let $V,W\in C(\R_+)$ be a viscosity subsolution and supersolution to \eqref{radial ergodic}, respectively. Furthermore, assume that $V\leq W$ on $\cA\cap(0,n-1)$. Then, we have that $V\leq W$ in $(0,n-1)$.
%\end{lem}
%\begin{proof}
%    hello moto
%\end{proof}

\begin{lem} \label{2nd radial CP}
    Assume \eqref{R1}--\eqref{R4}. Let $V,W\in C([0,\infty))$ be a viscosity subsolution and supersolution to \eqref{radial ergodic}, respectively, satisfying 
    \begin{equation} \label{growth rate of solns}
        \lim_{r\to \infty} \frac{V(r)}{r} = \lim_{r\to \infty} \frac{W(r)}{r} = c_F>0.
    \end{equation}
    Furthermore, assume that $V\leq W$ on $\cA_{rad}$. Then, we have that $V\leq W$ in $(0,\infty)$.
\end{lem}
\begin{proof}

Recall the definitions of $S_0,S_1,R_0,R_1$. We outline the major steps in the proof. 
\begin{enumerate}
    \item Prove the comparison principle on possible non-empty closed intervals with boundary lying in $\cA_{rad}$. In particular, the interval $(S_1,R_0)$.
    \item Prove the comparison principle on $(R_1,\infty)$
    \item Prove the comparison principle on $(0,S_0)$
\end{enumerate}

To that end, we begin with \emph{Step $1$}.

    \medskip \begin{bf}
        Step 1: 
    \end{bf}
    Let $a,b\in\cA_{rad}$ such that $(a,b)\not\subset\cA_{rad}$. By the hypothesis, we have that $V(a)\leq W(a)$ and $V(b)\leq W(b)$. Now suppose for contradiction that
    \begin{equation*}
        \max_{r\in[a,b]} \left\{ V(r) - W(r)\right\} >0.
    \end{equation*}
    By continuity, for $\lambda>1$ sufficiently close to one we have 
    \begin{equation*}
        \max_{r\in[a,b]} \left\{ V(r) - \lambda W(r)\right\} = V(r_1) - \lam W(r_1)>0,
    \end{equation*}
    for some $r_1 \in (a,b)$. Denote $W^\lambda(r) := \lambda W(r)$. We now use the doubling variables method and consider an auxiliary function $\Phi^\varepsilon:[a,b]^2\to\R$ defined by
    \begin{equation*}
        \Phi^\varepsilon(r,s): = V(r) - W^\lambda(s) - \frac{|r-s|^2}{\varepsilon}.
    \end{equation*}
    For $\varepsilon>0$ sufficiently small, we have a maximizer $(r^\varepsilon,s^\varepsilon)\in(a,b)^2$ of $\Phi^\varepsilon$ with a positive maximum. By compactness, we can extract a convergent subsequence, denoted by $\varepsilon$. In a similar manner as in the periodic case, we have that 
    \begin{equation} \label{conv subseq}
    \frac{|r^\varepsilon - s^\varepsilon|^2}{\varepsilon}\to 0 \quad \text{and}\quad (r^\varepsilon,s^\varepsilon)\to(\hat r,\hat r)\in(a,b)^2 \quad \text{up to a subsequence.}
    \end{equation}
    Fixing $s^\varepsilon$, we have the map $V(r) - \frac{|r-s^\varepsilon|^2}{\varepsilon}$ has a maximum at $r^\varepsilon$, and so by the viscosity subsolution test we have 
    \begin{equation} \label{ab sub test}
        \left(-\frac{(n-1)}{r^\varepsilon}\left(\frac{2(r^\varepsilon - s^\varepsilon)}{\varepsilon}\right) + \left|\frac{2(r^\varepsilon - s^\varepsilon)}{\varepsilon}\right|\right)_+\leq F(r^\varepsilon).
    \end{equation}
    Now fixing $r^\varepsilon$, we have $W^\lambda(s)-\left(\frac{-|r^\varepsilon - s|^2}{\varepsilon}\right)$ has a minimum at $s^\varepsilon$, and by the viscosity supersolution test we have 
    \begin{equation} \label{ab super test}
        \left(-\frac{(n-1)}{s^\varepsilon}\left(\frac{2(r^\varepsilon - s^\varepsilon)}{\varepsilon}\right) + \left|\frac{2(r^\varepsilon - s^\varepsilon)}{\varepsilon}\right|\right)_+\geq \lambda F(s^\varepsilon).
    \end{equation} \medskip
Denote $p^\varepsilon: = \frac{2(r^\varepsilon - s^\varepsilon)}{\varepsilon}$. Combining \eqref{ab sub test} and \eqref{ab super test} gives
\begin{equation} \label{combination}
    \left(-\frac{(n-1)}{r^\varepsilon}p^\varepsilon + \left|p^\varepsilon\right|\right)_+ - \left(-\frac{(n-1)}{s^\varepsilon}p^\varepsilon + \left|p^\varepsilon\right|\right)_+ \leq F(r^\varepsilon) - \lambda F(s^\varepsilon)
    %\left[\frac{(n-1)}{r^\varepsilon}\left(\frac{2(r^\varepsilon - s^\varepsilon)}{\varepsilon}\right) - \frac{(n-1)}{s^\varepsilon}\left(\frac{2(r^\varepsilon - s^\varepsilon)}{\varepsilon}\right) \right] \\
    %&+ \left[\left|\frac{2(r^\varepsilon - s^\varepsilon)}{\varepsilon}\right| - \left|\frac{2(r^\varepsilon - s^\varepsilon)}{\varepsilon}\right|\right]\\ 
\end{equation}
We claim that the left hand side of \eqref{combination} is non-negative. Indeed, if the second cutoff term is zero, then we have immediately the claim. Now suppose that the second cutoff term is positive. We then give a lower bound by estimating the first cutoff term by its argument. Explicitly, 
\begin{align*}
    \left(-\frac{(n-1)}{r^\varepsilon}p^\varepsilon + \left|p^\varepsilon\right|\right)_+ &- \left(-\frac{(n-1)}{s^\varepsilon}p^\varepsilon + \left|p^\varepsilon\right|\right)_+ \\
    &= \left(-\frac{(n-1)}{r^\varepsilon}p^\varepsilon + \left|p^\varepsilon\right|\right)_+ +\frac{(n-1)}{s^\varepsilon}p^\varepsilon - \left|p^\varepsilon\right| \\
    &\geq -\frac{(n-1)}{r^\varepsilon}p^\varepsilon + \left|p^\varepsilon\right| +\frac{(n-1)}{s^\varepsilon}p^\varepsilon - \left|p^\varepsilon\right| \\
    & = 2(n-1)\frac{(r^\varepsilon - s^\varepsilon)}{\varepsilon}\left(\frac{1}{s^\varepsilon} -\frac{1}{r^\varepsilon} \right) = \frac{2(n-1)}{r^\varepsilon s^\varepsilon}\frac{(r^\varepsilon - s^\varepsilon)^2}{\varepsilon}.
\end{align*}
Using the fact that we are considering a closed interval, we can estimate $r^\varepsilon s^\varepsilon \leq b^2$. Therefore, we have that 
\begin{equation*}
    \frac{2(n-1)}{b^2}\frac{(r^\varepsilon - s^\varepsilon)^2}{\varepsilon} \leq F(r^\varepsilon) - \lambda F(s^\varepsilon).
\end{equation*}
By \eqref{conv subseq} and the continuity of $F$, we know that 
\begin{equation*}
    0 \leq  (1-\lambda)F(\hat r)<0.
\end{equation*}
This is a contradiction, and so \emph{Step $1$} is proven, i.e., we must have that $V(r)\leq W(r)$ for all $r\in(a,b)$.

\medskip \begin{bf}
Step 2:
\end{bf}
By the hypothesis, we have that $V(R_1)\leq W(R_1)$. We claim that we must have that $V(r)\leq W(r)$ for all $r>R_1$. 
Suppose for contradiction that this is not the case, that is, there exists a radius $R_2>R_1$ such that 
\begin{equation*}
    V(R_2) - W(R_2) = \theta >0.
\end{equation*}
First note that for $\lambda>1$ close enough to one, we have 
\begin{equation} \label{first choice lam}
    V(R_2) - W^\lambda(R_2) \geq \frac{\theta}{4}.
\end{equation}
Since this is an unbounded domain, we make use of the equal linear growth rates of both $V$ and $W$ \eqref{growth rate of solns}. With our choice of $\lambda>1$, we also have
\begin{equation*}
    \lim_{r\to\infty} \frac{V(r)}{r} = c_F <\lambda c_F = \lim_{r\to \infty}\frac{W^\lambda(r)}{r}.
\end{equation*}
This, in turn, implies that we can find a $R_\infty>0$ large enough such that 
\begin{equation} \label{growth inequality}
    V(r) + 1 < W^\lambda(r) \text{ for all } r> R_\infty.
\end{equation}
We now consider the auxiliary function $\Phi^\varepsilon:[R_1,\infty)^2\to \R$ defined by
\begin{equation*}
    \Phi^\varepsilon(r,s) := V(r) - W^\lambda(s) - \frac{(r-s)^2}{\varepsilon}, 
\end{equation*}
for $\varepsilon>0$. Firstly, we see that 
\begin{equation*}
    \sup_{(r,s)\in[R_1,\infty)^2}\Phi^\varepsilon(r,s)\geq \Phi^\varepsilon(R_2,R_2) = V(R_2) - W^\lambda(R_2)\geq \frac{\theta}{4}>0,
\end{equation*}
for any $\varepsilon>0$. Now for $\varepsilon>0$ sufficiently small and using \eqref{growth inequality}, we have that 
\begin{equation*}
    \sup_{(r,s)\in[R_1,\infty)^2}\Phi^\varepsilon(r,s) = \sup_{(r,s)\in[R_1,R_\infty]^2}\Phi^\varepsilon(r,s) < +\infty
\end{equation*}
with maximizers $\{(r^\varepsilon, s^\varepsilon)\}_{\varepsilon}\subseteq (R_1,R_\infty]$. 
We can repeat the previous proof in Step $1$ for this region, and so we omit it. 

\medskip \begin{bf}
Step 3:  
\end{bf}
Now we consider the region $[0,S_0]$. By the hypothesis, we have that $V(S_0)\leq W(S_0)$. Suppose for contradiction, that 
\begin{equation*}
    \sup_{r\in[0,S_0]}\{ V(r) - W(r)\} = \sigma >0.
\end{equation*}
Again, by continuity we have for $\lambda>1$ sufficiently close to one that 
\begin{equation} \label{sup near origin}
    \sup_{r\in[0,S_0]} \{V(r) - \lambda W(r)\} \geq \frac{3\sigma}{4} >0.
\end{equation}
Then, there exists some $\tilde r\in [0,S_0)$ such that 
\begin{equation*}
    V(\tilde r) - \lambda W(\tilde r)\geq \frac{\sigma}{2}>0.
\end{equation*}
Since $V,W$ are continuous on $[0,S_0]$, we have that they are bounded. We then conclude that \eqref{sup near origin} is finite. Moreover, the auxiliary function 
\begin{equation*}
    \Psi(r,s): = V(r) - \lambda W(s) - \frac{(r - s)^2}{\varepsilon} - \delta\left(\frac{1}{r} + \frac{1}{s}\right) 
\end{equation*}
attains its maximum over $[0,S_0]$. Note that 
\begin{equation*}
    \Psi(\tilde r,\tilde r) = V(\tilde r) - \lambda W(\tilde r) -\delta\frac{2}{\tilde r}\geq \frac{\sigma}{2} - \delta\frac{2}{\tilde r}.
\end{equation*}
So, for $\delta>0$ sufficiently small, we have that 
\begin{equation*}
    \max_{[0,S_0]} \Psi(r,s) = \Psi(r^\varepsilon,s^\varepsilon)>0,
\end{equation*}
for all $\varepsilon>0$. We have that the maximizer $(r^\varepsilon,s^\varepsilon)$ satisfies
\begin{equation} \label{lower bound r}
    0< C_{\delta,\lambda }\leq r^\varepsilon
\end{equation}
and
\begin{equation} \label{lower bound for s}
0< C_{\delta,\lambda }\leq s^\varepsilon,
\end{equation}
for all $\varepsilon>0$. Observe that $\Psi(r^\varepsilon,s^\varepsilon)\geq \Psi(\tilde r,\tilde r)$. Explicitly, 
\[
V(r^\varepsilon) - \lambda W(s^\varepsilon) - \frac{(r^\varepsilon - s^\varepsilon)^2}{\varepsilon} - \delta\left(\frac{1}{r^\varepsilon} + \frac{1}{s^\varepsilon}\right)\geq V(\tilde r) - \lambda W(\tilde r) - \frac{2\delta}{\tilde r}.
\]
It follows then, that 
\[
\frac{(r^\varepsilon - s^\varepsilon)^2}{\varepsilon} \leq \left( V(r^\varepsilon) - V(\tilde r)\right) + \lambda\left(W(\tilde r) - W(s^\varepsilon)\right) +\frac{2\delta}{\tilde r}.
\]
Using the boundedness of $V,W$, we get 
\begin{equation*}
    \frac{(r^\varepsilon - s^\varepsilon)^2}{\varepsilon} \leq 2\left(\|V\|_{\Li([0,S_0])} + \lambda\|W\|_{\Li([0,S_0])}\right) +\frac{2\delta}{\tilde r}\leq C.
\end{equation*}
Hence, $|r^\varepsilon - s^\varepsilon|\leq C\varepsilon^{1/2}$.

Now, observe that $\Psi(r^\varepsilon,s^\varepsilon)\geq \Psi(r^\varepsilon,r^\varepsilon)$. Explicitly, 
\[
V(r^\varepsilon) - \lambda W(s^\varepsilon) - \frac{(r^\varepsilon - s^\varepsilon)^2}{\varepsilon} - \delta\left(\frac{1}{r^\varepsilon} + \frac{1}{s^\varepsilon}\right)\geq V(r^\varepsilon) - \lambda W(r^\varepsilon) - \frac{2\delta}{r^\varepsilon}.
\]
It follows then, that 
\begin{align*}
\frac{(r^\varepsilon - s^\varepsilon)^2}{\varepsilon} &\leq \lambda\left(W(r^\varepsilon) - W(s^\varepsilon)\right) +\delta\left(\frac{1}{r^\varepsilon} - \frac{1}{s^\varepsilon}\right)\leq \lambda \omega(|r^\varepsilon -s^\varepsilon|) + \frac{\delta}{r^\varepsilon s^\varepsilon}|r^\varepsilon - s^\varepsilon| \\
&\leq \lambda \omega(|r^\varepsilon -s^\varepsilon|) + \frac{C\delta}{(C_{\delta,\lambda})^2}\varepsilon^{1/2}, 
\end{align*}
where $\omega$ is the modulus of continuity of W. We send $\varepsilon\to 0$ and, hence, we have 
\begin{equation} \label{nice convergence}
    \lim_{\varepsilon\to 0} \frac{(r^\varepsilon - s^\varepsilon)^2}{\varepsilon} = 0. 
\end{equation} By compactness, we can extract a convergent subsequence, still denoted by the full sequence, such that 
\begin{equation}
    (r^\varepsilon,s^\varepsilon) \to (\hat r,\hat r)\in(0,S_0)^2. 
\end{equation}

We now apply the viscosity subsolution and supersolution tests. For clarity, we denote $p^\varepsilon: = \frac{2(r^\varepsilon - s^\varepsilon)}{\varepsilon}$, $q_{r^\varepsilon}: =\frac{\delta}{({r^\varepsilon})^2}$, and $q_{s^\varepsilon}: = \frac{\delta}{(s^\varepsilon)^2}$. Then, if $V$ and $W$ were smooth, we would have at a maximum 
\begin{equation*}
    DV(r^\varepsilon) = p^\varepsilon - q_{r^\varepsilon},
\end{equation*}
and 
\begin{equation*}
    D[\lambda W(s^\varepsilon)] = p^\varepsilon + q_{s^\varepsilon}.
\end{equation*}
Applying the subsolution and supersolution tests gives
\begin{equation*}
    \left(
    -\frac{(n-1)}{r^\varepsilon}[p^\varepsilon - q_{r^\varepsilon}] + |p^\varepsilon - q_{r^\varepsilon}|
    \right)_+ 
    \leq
    F(r^\varepsilon),
\end{equation*}
and 
\begin{equation*}
    \left(
    -\frac{(n-1)}{s^\varepsilon}[p^\varepsilon + q_{s^\varepsilon}] + |p^\varepsilon + q_{s^\varepsilon}|
    \right)_+
    \geq
    \lambda F(s^\varepsilon).
\end{equation*}
Using the above tests, we see that 
\begin{align} \label{combo for step 3}
    \left(
    -\frac{(n-1)}{r^\varepsilon}[p^\varepsilon - q_{r^\varepsilon}] + |p^\varepsilon - q_{r^\varepsilon}|
    \right)_+  
    &- 
    \left(
    -\frac{(n-1)}{s^\varepsilon}[p^\varepsilon + q_{s^\varepsilon}] + |p^\varepsilon + q_{s^\varepsilon}|
    \right)_+ \\
    &\leq
        F(r^\varepsilon) - \lambda F(s^\varepsilon). \notag
\end{align}
We tackle the more trivial case first. That is, the second cutoff term is zero. Then, it is clear that the left hand side of \eqref{combo for step 3} is non-negative. Then, sending $\varepsilon\to 0$ up to the convergent subsequence gives our contradiction. Thus, we consider the non-trivial case, where the second cutoff term is strictly positive. Let us consider the left hand side of \eqref{combo for step 3}. 
\begin{align*}
\left(
    -\frac{(n-1)}{r^\varepsilon}[p^\varepsilon - q_{r^\varepsilon}] + |p^\varepsilon - q_{r^\varepsilon}|
    \right)_+  
    &+ 
    \frac{(n-1)}{s^\varepsilon}[p^\varepsilon + q_{s^\varepsilon}] - |p^\varepsilon + q_{s^\varepsilon}| \\
    &\geq 
    -\frac{(n-1)}{r^\varepsilon}[p^\varepsilon - q_{r^\varepsilon}] + |p^\varepsilon - q_{r^\varepsilon}| \\
    &+ 
    \frac{(n-1)}{s^\varepsilon}[p^\varepsilon + q_{s^\varepsilon}] - |p^\varepsilon + q_{s^\varepsilon}|.
\end{align*}
This implies that 
\begin{align}\label{almost contradiction}
    0 
    \leq
    [F(r^\varepsilon) - \lambda F(s^\varepsilon)] &+ (n-1)p^\varepsilon\left(\frac{1}{r^\varepsilon} - \frac{1}{s^\varepsilon}\right) - (n-1)\left( \frac{q_{s^\varepsilon}}{s^\varepsilon} + \frac{q_{r^\varepsilon}}{r^\varepsilon} \right) \\
    & + \left[ |p^\varepsilon + q_{s^\varepsilon}| -  |p^\varepsilon -q_{r^\varepsilon}| \right] = :A + B + C + D.  \notag
\end{align}
For $A$, we send $\varepsilon\to0$ up to a subsequence, which gives 
\begin{equation} \label{A result}
    \lim_{\varepsilon\to 0} A = (1-\lambda) F(\hat r) <0. 
\end{equation}

Next, we consider $B$ and observe that
\begin{equation*}
    |B| = (n-1)|p^\varepsilon|\left| \frac{1}{r^\varepsilon} - \frac{1}{s^\varepsilon}\right|\leq \frac{2(n-1)}{r^\varepsilon s^\varepsilon} \frac{(r^\varepsilon - s^\varepsilon)^2}{\varepsilon} 
    \leq 
    \frac{2(n-1)}{(C_{\delta,\lambda})^2} \frac{(r^\varepsilon - s^\varepsilon)^2}{\varepsilon},
\end{equation*}
where we used \eqref{lower bound r} and \eqref{lower bound for s} to derive the last estimate. Lastly, using \eqref{nice convergence} we see that 
\begin{equation*} 
\lim_{\varepsilon\to0} B = 0.
\end{equation*}
We now deal with $C$. It follows that 
\begin{equation*}
    |C| = (n-1)\left(\frac{\delta}{(s^\varepsilon)^3} + \frac{\delta}{(r^\varepsilon)^3} \right). 
\end{equation*}
Send $\varepsilon \to 0$ up to a subsequence and we get 
\begin{equation*}
    \lim_{\varepsilon \to 0} |C| \leq \frac{2(n-1)}{(\hat r)^3} \delta. 
\end{equation*}
Sending $\delta\to 0$ afterwards gives the desired result. 
Lastly, consider 
\begin{align*}
    |D| &= \left| |p^\varepsilon + q_{s^\varepsilon}| -  |p^\varepsilon -q_{r^\varepsilon}| \right| \leq \left| p^\varepsilon + q_{s^\varepsilon} - p^\varepsilon + q_{r^\varepsilon}\right| = \left| q_{s^\varepsilon} + q_{r^\varepsilon}\right| \\
    & = \left| \frac{\delta}{(s^\varepsilon)^2} + \frac{\delta}{(r^\varepsilon)^2} \right| \notag
\end{align*}
where we used the triangle inequality. Then, we see that 
\begin{align*}
    \lim_{\varepsilon\to0}|D|\leq \frac{2\delta}{(\hat r)^2},
\end{align*}
up to a subsequence. Next, sending $\delta\to0$ implies our desired result. Combining these results for $A,B,C,$ and $D$ gives a contradiction, which finishes \emph{Step $3$}.
Hence, we must have the Comparison Principle for $(0,\infty)$. 
\end{proof}

We define the \emph{half-relaxed} limits $U^+,U^-\in C([0,\infty))$ of the unique solution to the Cauchy problem \eqref{radcauchy} by 
\begin{equation*}
    U^+(r): = \limsups\limits_{t\to\infty}{U(r,t)} = \lim\limits_{t\to\infty}\sup\left\{U(\rho,s): |r-\rho|<1/s, s\geq t\right\}
\end{equation*}
and
\begin{equation*}
    U^-(r): = \liminfs\limits_{t\to\infty}{U(r,t)} = \lim\limits_{t\to\infty}\inf \left\{U(\rho,s):|r- \rho|<1/s,s\geq t \right\}.
\end{equation*}
\begin{proof}[Proof of Theorem \ref{lt behavior rad}.]

Observe that by definition, we have that $U^-\leq U^+$ for all $r\in(0,\infty)$. Moreover, by stability of the viscosity solution, we have that $U^+$ and $U^-$ are a subsolution and supersolution to \eqref{radial ergodic}, respectively. Using Lemma \ref{2nd radial CP}, we must have that $U^+\leq U^-$, which clearly gives $U^+ = U^-$.
By stability of viscosity solutions, $V: = U^+ = U^-$ is a solution to the ergodic problem \eqref{radial ergodic}. Hence, the theorem is proven true.
\end{proof}

\begin{rem}
The class of $C([0,\infty))$ subsolutions and supersolutions with linear growth that are suitable for our proof of the Comparison Principle in Lemma \ref{2nd radial CP} is non-empty. Indeed, we are able to explicitly construct a classical solution with linear growth via a similar strategy used in \cite{remarks}. 
\end{rem}

    \begin{proof}[Proof of Theorem \ref{Thm Rep formula for limit}]
       Using the fact \eqref{monotonicity}, we have that for $r\in\cA_{rad}$
        \begin{equation*}
            U(r,t) \to \inf_{\gamma\in\mathcal{C}(t,0;r,s),s\in(0,\infty),t>0}\left\{ 
\int_0^t F(\gamma(z))\,dz+ G(s)\right\}= v_G(r),
        \end{equation*}
        as $t\to \infty$.
    In light of the uniqueness property of $\cA_{rad}$ and the convexity of the Hamiltonian, the limiting equation is represented by \eqref{rep formula for limit}. We refer the reader to Corollary $5.25$ in \cite{dynamical}, for more details on this result.   
    \end{proof}
    \medskip

%%%%%%%%%%%%%%%%%%%%%%%%%%%%%%%%%%%%%%%%%%%%%%%%%%%%%%%%%%%%%%%%%

\section{Regularity of Viscosity Solutions}

\subsection{Local Lipschitz Regularity in the Periodic Setting} In order to utilize an argument that requires spatial and temporal invariance of the Hamiltonian, we set the \emph{wind flow} vector field to zero, i.e., we assume
\begin{equation}
    \vec W(x) \equiv \vec 0.
\end{equation}
That is, we study the regularity of 
\begin{equation} \tag{N} \label{periodic no wind}
\left\{
\begin{aligned}
    u_t + \left(- a^{ij}(Du)(D^2u)_{ij} +|Du|\right)_+ =f(x) \quad &\text{ in } \T^n\times(0,\infty)\\
    u(x,0) = g(x) \quad &\text{ on } \T^n,
    \end{aligned}
\right.
\end{equation}
with $f\in \Lip(\T^n)$ and $g\in C^2(\T^n)$.

\lem{Assume $f\in \Lip(\T^n)$, $g\in C^2(\T^n)$. Let $u\in C(\T^n\times(0,\infty))$ be the unique viscosity solution to 
\eqref{periodic no wind}. For any $T\in(0,\infty)$, we have that $u\in\Lip(\T^n\times [0,T])$.
}
\proof{} Let us first show that $u$ is Lipschitz in space. To that end, we construct a suitable viscosity subsolution to \eqref{periodic no wind}. Fix some $y\in\T^n$ and $t\in[0,T]$. Consider the function
\begin{equation*}
\varphi(x,t): = u(x-y,t)-C_0|y|t -C_1|y|,
\end{equation*}
with constants $C_0,C_1>0$ to be chosen later. We denote the Lipschitz constants of the source, $f$, and initial data, $g$, by $L_f$ and $L_g$, respectively. Observe that $\varphi$ is a viscosity subsolution to \eqref{periodic no wind}, for appropriately chosen constants $C_0, C_1$. Indeed, we have 
\begin{align*}
    \varphi(x,0) &= u(x-y,0)- C_1|y| = g(x-y)-C_1|y|\\
    & = g(x-y)-g(x) +g(x) -C_1|y|\leq L_g|y| +g(x) -C_1|y|\\
    & = g(x), 
\end{align*}
for $C_1\geq L_g$. Consider 
\begin{align*}
    \varphi_t&(x,t) + \left( -a^{ij}(D\varphi(x,t))[D^2\varphi(x,t)]_{ij} + |D\varphi(x,t)| \right)_+ \\
    &=u_t(x-y,t) - C_0|y| + \left( -a^{ij}(Du(x-y,t))[D^2u(x-y,t)]_{ij} + |Du(x-y,t)| \right)_+ \\
    & = f(x-y) - C_0|y| = f(x-y) - f(x) + f(x) - C_0|y| \\
    & \leq L_f|y| - C_0|y| +f(x) \leq f(x),
\end{align*}
where we used the fact that spatial translations of $u(x,t)$ are still viscosity solutions to \eqref{periodic no wind} at a translated point and chose the constant such that $C_0\geq L_f$. It is clear that $\varphi$ is a viscosity subsolution to \eqref{periodic no wind}. We now create the barrier function from above. Consider the function 
\begin{equation*}
    \psi(x,t): = u(x-y,t) + C_0|y|t + C_1|y|. 
\end{equation*}
We see that in a similar fashion that 
\begin{equation*}
    \psi(x,0) \geq g(x)
\end{equation*}
and 
\begin{equation*}
    \psi_t(x,t) + \left( -a^{ij}(D\psi(x,t))[D^2\psi(x,t)]_{ij} + |D\psi(x,t)| \right)_+ \geq f(x),
\end{equation*}
due to our choices of $C_0$ and $C_1$. We conclude that $\psi$ is a viscosity supersolution to \eqref{periodic no wind}. By the usual Comparison Principle for \eqref{periodic no wind}, we have that
\begin{equation*}
    \varphi(x,t)\leq u(x,t)\leq\psi(x,t),
\end{equation*}
for all $(x,t)\in \T^n\times[0,T]$. Therefore, we have 
\begin{equation} \label{space Lip}
    |u(x,t)-u(x-y,t)|\leq (C_0t+C_1)|y|\leq (C_0T +C_1)|y|. 
\end{equation}
Hence, we conclude that $u$ is Lipschitz in space. Note, however, that the estimate \eqref{space Lip} depends on $T$. We make a note on this later, see Remark \ref{not global}. \medskip

Now we show that $u$ is Lipschitz in time. Again, we will use barrier functions and the Comparison Principle along with the temporal invariance of the Hamiltonian in \eqref{periodic no wind} to conclude this. Consider the barrier functions
\begin{equation*}
    \phi(x,t):= -\left(\|f\|_{\Li(\T^n)} + C\right)t + g(x)
\end{equation*}
and 
\begin{equation*}
    \Phi(x,t): = \left( \|f\|_{\Li(\T^n)} + C \right) t+g(x). 
\end{equation*}
Here, we choose the constant $C>0$ satisfying
\begin{equation}\label{bound on initial nonlinearity}
    \left( -a^{ij}(Dg(x))[D^2g(x)]_{ij} + |Dg(x)| \right)_+ \leq C,
\end{equation}
for all $x\in\T^n$. 
Clearly, we have
\begin{equation*}
    \phi(x,0) = g(x) = \Phi(x,0).
\end{equation*}
Now observe that
\begin{align*}
    \phi_t(x,t) &+ \left( -a^{ij}(D\phi(x,t))[D^2\phi(x,t)]_{ij} + |D\phi(x,t)| \right)_+ \\
    &= -\left(\|f\|_{\Li(\T^n)} + C \right) +\left( -a^{ij}(Dg(x))[D^2g(x)]_{ij} + |Dg(x)| \right)_+ \\
    & \leq -\|f\|_{\Li(\T^n)} \leq f(x).
\end{align*}
Therefore, $\phi$ is a viscosity subsolution to \eqref{periodic no wind}. On the other hand, 
\begin{align*}
    \Phi_t(x,t) &+ \left( -a^{ij}(D\Phi(x,t))[D^2\Phi(x,t)]_{ij} +|D\Phi(x,t)| \right)_+ \\
    & = \left( \|f\|_{\Li(\T^n)} + C \right) +\left( -a^{ij}(Dg(x))[D^2g(x)]_{ij} + |Dg(x)| \right)_+ \\
    & \geq \|f\|_{\Li(\T^n)} \geq f(x). 
\end{align*}
Note that we do not need to use \eqref{bound on initial nonlinearity} for this step as the operator is non-negative definite. 
We conclude that $\Phi$ is a viscosity supersolution to \eqref{periodic no wind}. The Comparison Principle implies that 
\begin{equation*}
    \phi(x,t)\leq u(x,t)\leq\Phi(x,t),
\end{equation*}
for all $(x,t)\in \T^n\times[0,T]$. This directly gives
\begin{equation*}
    \frac{|u(x,t) - u(x,0)|}{t}\leq \left( \|f\|_{\Li(\T^n)} + C \right), 
\end{equation*}
and so,
\begin{equation} \label{lip at time zero}
    \sup_{t\geq0}\left\{\frac{|u(x,t)-u(x,0)|}{t}\right\}\leq \left( \|f\|_{\Li(\T^n)} + C\right).
\end{equation}
Therefore, $u$ is Lipschitz in time at $t=0$. Now fix some $s>0$, and let $v(x,t):=u(x,t+s)$. Since the PDE is invariant in time, $v$ is another viscosity solution to \eqref{periodic no wind} with different initial data $v_0(x): =v(x,0)=u(x,s)$. Denote $u_0(x): = u(x,0)$. Observe that 
\begin{equation*}
    v_0 - \|u_0-v_0\|_{\Li(\T^n)} \leq u_0 \leq v_0 + \|u_0-v_0\|_{\Li(\T^n)}.
\end{equation*}
This, and the fact that the \eqref{periodic no wind} does not \emph{see} additive constants implies that 
\begin{equation*}
    v - \|u_0-v_0\|_{\Li(\T^n)}
\end{equation*}
and 
\begin{equation*}
    v + \|u_0-v_0\|_{\Li(\T^n)}
\end{equation*}
are a viscosity subsolution and supersolution to \eqref{periodic no wind}, respectively. Then, by the usual Comparison Principle we have that 
\begin{equation*}
    v(x,t) - \|u_0-v_0\|_{\Li(\T^n)}\leq u(x,t)\leq v(x,t) + \|u_0-v_0\|_{\Li(\T^n)},
\end{equation*}
for all $(x,t)\in \T^n\times(0,\infty)$. It follows that 
\begin{align*}
    \frac{|u(x,t+s) - u(x,t)|}{s} &\leq \left\|\frac{u_0-v_0}{s}\right\|_{\Li(\T^n)} 
    = \left\|\frac{u(\cdot,s)-u(\cdot,0)}{s}\right\|_{\Li(\T^n)}\\
    &\leq \left( \|f\|_{\Li(\T^n)} + C \right) .
\end{align*}
This holds for all $s>0$, and so we conclude that $u$ is globally Lipschitz in time. \qed

A few remarks are in order.

\begin{rem} \label{not global}
Due to our spatial gradient bound being dependent on time, we are only able to show local Lipschitz regularity. As $T\to\infty$, our estimate \eqref{space Lip} becomes immaterial. 
\end{rem}

\begin{rem}
    In order to achieve a global regularity result, we will require different methods. For example, it may be fruitful to use the Viscosity Bernstein method used in \cite{vis bernstein} with some modifications as the non-coercivity and lack superlinearity of the Hamiltonian creates challenges. It also may be possible that the unique viscosity solution to \eqref{periodic cauchy} is indeed not globally Lipschitz, and a counterexample could be found. This will be a future direction of research related to this problem. 
\end{rem}

\subsection{Lipschitz Regularity in Radially Symmetric Setting}
We now discuss the regularity of viscosity solution $U$ to \eqref{radcauchy}. Our main tool to answer this regularity question is the representation formula \eqref{refined rep formula}, which we state below for convenience:
\begin{align*}
    U(r,t) = \inf_{\gamma\in\mathcal{C}(t,0;r,s),s\in(0,\infty)}\left\{ 
\int_0^t F(\gamma(z))\,dz+ G(\gamma(0))\right\}.
\end{align*}

We now state the local regularity result:
\begin{lem}
    Assume that $F\in C^1((0,\infty))\cap \Li((0,\infty))$ and $G\in C^1((0,\infty))$. Let $U\in C([0,\infty)\times(0,\infty))$ be the viscosity solution to \eqref{radcauchy}. Then, $U$ is globally Lipschitz in time, and, additionally, we have that $U\in\Lip((\alpha,\infty)\times(0,\infty))$ for any $\alpha>n-1$.
\end{lem}
\begin{proof}
    We first show that $U$ is globally Lipschitz in time. To that end, consider two fixed times $t_1,t_2\in(0,\infty)$ and without loss of generality we assume $t_1<t_2$.
    By the representation formula we have
    \begin{align*}
        | U(r,t_2) - U(r,t_1) | &= 
        \left|
        \inf_{\gamma\in\mathcal{C}(t_2,0;r,s),s\in(0,\infty)} \left\{ \int_0^{t_2} F(\gamma(z))\,dz + G(\gamma(0)) \right\} \right. \\
         & \qquad \qquad \qquad - 
         \left. \inf_{\gamma\in\mathcal{C}(t_1,0;r,s),s\in(0,\infty)} \left\{ \int_0^{t_1} F(\gamma(z))\,dz + G(\gamma(0)) \right\}
        \right|.
    \end{align*}
    Fix $\varepsilon>0$, and choose a trajectory $\gamma^*$ with initial starting point $s^*$ such that 
    \begin{equation}
        -U(r,t_1) \leq \varepsilon - \int_0^{t_1} F(\gamma^*(s))\,ds - G(\gamma^*(0))
    \end{equation}
    Now observe
    \begin{align*}
        | U(r,t_2) - U(r,t_1) | &= 
        \left|
        \inf_{\gamma\in\mathcal{C}(t_2,0;r,s),s\in(0,\infty)} \left\{ \int_0^{t_2} F(\gamma(z))\,dz + G(\gamma(0)) \right\} \right. \\
         & \qquad \qquad \qquad - 
         \left. \inf_{\gamma\in\mathcal{C}(t_1,0;r,s),s\in(0,\infty)} \left\{ \int_0^{t_1} F(\gamma(z))\,dz + G(\gamma(0)) \right\}
        \right|\\
        & \leq \left| 
        \inf_{\gamma\in\mathcal{C}(t_2,0;r,s),s\in(0,\infty)} \left\{ \int_0^{t_2} F(\gamma(z))\,dz - G(\gamma(0)) \right\} \right.\\ 
        & 
        \left. \qquad \qquad \qquad - \int_0^{t_1} F(\gamma^*(s))\,ds - G(\gamma^*(0)) + \varepsilon
        \right| \\
        & \leq \left| 
        \inf_{\gamma\in\mathcal{C}(t_2,0;r,s),s\in(0,\infty)} \left\{ \int_0^{t_2} F(\gamma(z))\,dz + G(\gamma(0)) \right\} \right.\\ 
        & 
        \left. \qquad \qquad \qquad - \int_0^{t_1} F(\gamma^*(s))\,ds - G(\gamma^*(0))
        \right| + \varepsilon.
    \end{align*}
    For the first term on the right hand side, we will choose the sub-optimal policy in which we follow the near optimal trajectory $\gamma^*$ until time $t_1$, then we remain there at $r$ until time $t_2$. Explicitly, 
    \begin{align*}
        | U(r,t_2) - U(r,t_1) |&\leq \left| 
        \int_0^{t_1} F(\gamma^*(s))\,ds + \int_{t_1}^{t_2} F(r)\,ds + G(\gamma^*(0)) \right. \\
        &\left. \qquad \qquad \qquad \qquad - \int_0^{t_1} F(\gamma^*(s))\,ds - G(\gamma^*(0))
        \right| + \varepsilon\\
        & = \left|
        F(r)\int_{t_1}^{t_2}\,ds
        \right| \leq \|F\|_{\Li((0,\infty))}|t_2 - t_1| +\varepsilon. 
    \end{align*}
    Letting $\varepsilon\to 0$, we conclude that $U$ is Lipschitz in time, i.e., 
    \begin{equation} \label{a priori est time}
        \|U_t\|_{\Li((0,\infty)\times(0,\infty))} \leq C.
    \end{equation}
    
    We point our focus to the spatial regularity and use the coercivity of the Hamiltonian $H(r,p)$ in the region $(n-1,\infty)$. Indeed, let $\alpha>n-1$ and $r \in (\alpha,\infty)$, then observe 
    \begin{align*}
        H(r,p) &= \left( -\frac{(n-1)}{r} p +|p| \right)_+ - F(r) = -\frac{(n-1)}{r} p +|p| - F(r) \\
        &\geq - \frac{(n-1)}{\alpha}|p| +|p| - \|F\|_{\Li((0,\infty))}
        = \left( 1-\frac{(n-1)}{\alpha} \right)|p| - \|F\|_{\Li((0,\infty))}.
    \end{align*}
    This along with \eqref{a priori est time} and the PDE in \eqref{radcauchy}, gives
    \begin{equation} \label{a priori est space}
        \|U_r\| _{\Li((\alpha,\infty)\times(0,\infty))} \leq \frac{C}{\left(1-\frac{(n-1)}{\alpha}\right)}.
    \end{equation}
\end{proof}

\begin{rem}
    Observe that as $\alpha\to n-1$, the estimate \eqref{a priori est space} becomes useless, and provides some insight to the delicate nature of the viscosity solution $U$ on the interface of the regions $(0,n-1]$ and $(n-1,\infty)$. The loss of coercivity of the Hamiltionian in $(0,n-1]$ asserts this sensitivity. This is an interesting question that requires more delicate analysis of the solution using its representation formula.
\end{rem}

\section*{Acknowledgments}

The author would like to thank Professors Hung V. Tran and Hiroyoshi Mitake for their invaluable feedback and suggestions. Moreover, the comments provided by Jiwoong Jang are greatly appreciated. 

%%%%%%%%%%%%%%%%%%%%%%%%%%%%%%%%%%%%%%%%%%%%%%%%%%%%%%%%%%%%%%%%%

\end{document}